\documentclass{amsart}

\usepackage{amsfonts}
\usepackage{amssymb}
\usepackage{amstext}
\usepackage{amsmath}
\usepackage{amsthm,verbatim}
\usepackage{xspace}
\usepackage{graphicx}
\usepackage{algorithm}
\newtheorem{thm}{Theorem}[section]
\newtheorem{cor}[thm]{Corollary}

\theoremstyle{definition}

\theoremstyle{remark}

\numberwithin{equation}{section}

\newcommand{\eps}{\varepsilon}
\newcommand{\To}{\longrightarrow}

\begin{document}

\title{On the distribution of random words in a compact Lie group}%
\author{Hariharan Narayanan}
\address{School of Technology and Computer Science, Tata Institute of Fundamental Reseaarch, Mumbai}
\email{hariharan.narayanan@tifr.res.in}


\newtheorem{theorem}{Theorem}
\newtheorem{proposition}{Proposition}
\newtheorem{lemma}[theorem]{Lemma}
\newtheorem{corollary}[theorem]{Corollary}
\newtheorem{remark}[theorem]{Remark}
\newtheorem{question}[theorem]{Question}
\newtheorem{example}[theorem]{Example}
\newtheorem{definition}{Definition}
\newtheorem{notation}{Notation}
\newtheorem{obs}{Observation}
\newtheorem{fact}{Fact}
\newtheorem{claim}{Claim}
\newcommand{\Prob}{\ensuremath{\text{Pr}}}
\newcommand{\E}{\ensuremath{\mathbb E}}
\newcommand{\R}{\ensuremath{\mathbb R}}
\newcommand{\cX}{\ensuremath{\mathcal X}}
\newcommand{\cT}{\ensuremath{\mathcal T}}
\newcommand{\K}{\ensuremath{\mathcal K}}
\newcommand{\F}{\ensuremath{\mathcal F}}
\newcommand{\FF}{\ensuremath{\mathcal F}}
\newcommand{\tr}{\ensuremath{{\scriptscriptstyle\mathsf{Tr}}}}
\newcommand{\inner}[1]{\left\langle #1 \right\rangle}
\newcommand{\reach}{\mathrm{reach}}

\newcommand{\lab}{\label} \newcommand{\M}{\ensuremath{\mathcal M}} \newcommand{\ra}{\ensuremath{\rightarrow}} \def\eee{{\mathrm e}} \def\a{{\mathbf{\alpha}}} \def\de{{\mathbf{\delta}}} \def\De{{{\Delta}}} \def\l{{\mathbf{\lambda}}} \def\m{{\mathbf{\mu}}}
\def\tm{{\tilde{\mu}}} \def\var{{\mathrm{var}}} \def\beq{\begin{eqnarray}} \def\eeq{\end{eqnarray}} \def\ben{\begin{enumerate}}
\def\een{\end{enumerate}}
 \def\bit{\begin{itemize}}
\def\eit{\end{itemize}}
 \def\beqs{\begin{eqnarray*}} \def\eeqs{\end{eqnarray*}} \def\bel{\begin{lemma}} \def\eel{\end{lemma}}
\newcommand{\N}{\mathbb{N}} \newcommand{\Z}{\mathbb{Z}} \newcommand{\Q}{\mathbb{Q}} \newcommand{\C}{\mathcal{C}} \newcommand{\CC}{\mathcal{C}}
\newcommand{\T}{\mathbb{T}} \newcommand{\A}{\mathbb{A}} \newcommand{\x}{\mathbf{x}} \newcommand{\y}{\mathbf{y}} \newcommand{\z}{\mathbf{z}}
\newcommand{\n}{\mathbf{n}} \newcommand{\I}{\mathbb{I}} \newcommand{\II}{\mathcal{I}} \newcommand{\EE}{\mathbb{E}} \newcommand{\p}{\mathbb{P}}
\newcommand{\PP}{\mathcal P} \newcommand{\BB}{\mathcal B} \newcommand{\HH}{\mathcal H} \newcommand{\e}{\mathbf{e}} \newcommand{\one}{\mathrm{1}}
\newcommand{\LL}{\mathcal L} \newcommand{\MM}{\mathcal M}\newcommand{\NN}{\mathcal N} \newcommand{\la}{\lambda} \newcommand{\Tr}{\text{Trace}} \newcommand{\aaa}{\alpha}
\def\eee{{\mathrm e}} \def\a{{\mathbf{\alpha}}} \def\l{{\mathbf{\lambda}}} \def\eps{{\epsilon}} \def\A{{\mathcal{A}}} \def\ie{i.\,e.\,} \def\g{G}
\def\vol{\mathrm{vol}}\newcommand{\tc}{{\tilde{c}}}
\newcommand{\tP}{\tilde{P}}
\newcommand{\sig}{\sigma}
\newcommand{\rr}{\mathbf{r}}
\renewcommand{\a}{\alpha}
\newcommand{\ta}{\tilde{\alpha}}
\renewcommand{\b}{\beta}
\newcommand{\tf}{\hat{F}}
\newcommand{\tix}{\tilde{x}}
\newcommand{\tif}{\tilde{f}}
\newcommand{\tih}{\tilde{h}}
\newcommand{\tiv}{\tilde{v}}
\newcommand{\tG}{\hat{G}}
\newcommand{\tZ}{\hat{Z}}
\newcommand{\hz}{\hat{z}}
\newcommand{\tg}{\hat{G}}
\newcommand{\tv}{\hat{V}}
\newcommand{\tmu}{\tilde{\mu}}
\newcommand{\tS}{\tilde{S}}
\newcommand{\tk}{\hat{K}}
\newcommand{\rmm}{\rho_{min}}
\newcommand{\La}{\Lambda}
\newcommand{\sy}{\mathcal{S}_y}
\newcommand{\cc}{\mathbf{c}}
\newcommand{\RR}{\mathbb{R}}
\newcommand{\dist}{dist}
\newcommand{\dhaus}{\mathbf{d}_{\mathtt{haus}}}
\newcommand{\G}{\mathcal{G}}
\newcommand{\TG}{\mathcal{\tilde{G}}}
\newcommand{\fat}{\mathrm{fat}}
\newcommand{\cuf}{\textsf{CurvFit}}
\renewcommand{\H}{\mathbb{H}}
\newcommand{\cscs}{{\textsf{Coreset Conditions}\,}}
\newcommand{\csone}{{\textsf{Case 1}\,}}
\newcommand{\cstwo}{{\textsf{Case 2}\,}}
\newcommand{\lareig}{{\textsf{LarEig}}}
\newcommand{\smpl}{{\textsf{Sample}}}
\newcommand{\inprd}{{\textsf{InnProd}}}
\newcommand{\cyl}{{\texttt{cyl}}}
\newcommand{\tran}{{\texttt{tran}}}
\newcommand{\norm}{{\texttt{norm}}}
\newcommand{\D}{\mathcal{\bar{D}}}
\newcommand{\base}{\texttt{base}}
\newcommand{\stalk}{\texttt{stalk}}
\newcommand{\se}{\mathrm{s}}
\newcommand{\htau}{\hat{\tau}}
\newcommand{\str}{\texttt{strt}}
\newcommand{\pare}{p} \newcommand{\chil}{h}
\newcommand{\hy}{\hat{y}}
\newcommand{\hV}{\hat V}
\newcommand{\ty}{\tilde{y}}
\newcommand{\ts}{\tilde{s}}
\newcommand{\tga}{\tilde{\gamma}}
\newcommand{\tx}{\tilde{x}}
\newcommand{\hN}{\hat{N}}
\newcommand{\oc}{\overline{c}}
\newcommand{\ob}{\check{c}}
\newcommand{\oC}{\overline{C}}
\newcommand{\nli}{\\\\\noindent}
\newcommand{\iN}{\mathit{N}}
\newcommand{\hess}{\,\texttt{Hess}\,}
\newcommand{\nF}{\nabla_F}
\newcommand{\lip}{Lipschitz}
\newcommand{\beps}{\bar{\eps}}
\newcommand{\ttp}{\texttt{P}}
\newcommand{\wh}{\texttt{Wh}}
\newcommand{\bn}{\bar{n}}
\newcommand{\bN}{\bar{N}}
\newcommand{\obj}{\zeta}
\newcommand{\bmp}{\theta}
\newcommand{\hak}{\hat{k}}
\newcommand{\asdf}{{asdf}}
\newcommand{\grid}{\texttt{grid}}
\newcommand{\fix}{\marginpar{FIX}}
\newcommand{\new}{\marginpar{NEW}}
\newcommand{\beqn}{\begin{equation}}
\newcommand{\eeqn}{\end{equation}}

\begin{abstract}
Let $G$ be a compact Lie group. Suppose $g_1, \dots, g_k$ are chosen independently from the Haar measure on $G$. Let  $\A = \cup_{i \in [k]} \A_i$, where, $\A_i :=  \{g_i\} \cup \{g_i^{-1}\}$.  Let $\mu_{\A}^\ell$ be the uniform measure over all words of length $\ell$ whose alphabets belong to $\A$. We give probabilistic bounds on the nearness of a heat kernel smoothening of $\mu_{\A}^\ell$ to a constant function on $G$ in $\LL^2(G)$. We also give probabilistic bounds on the maximum distance of a point in $G$ to the support of $\mu_{\A}^\ell$. Lastly, we show that these bounds cannot in general be significantly improved by analyzing the case when $G$ is the $n-$dimensional torus.

The question of a spectral gap of a natural Markov operator associated with $\A$ when $G$ is $SU_2$  was reiterated by Bourgain and Gamburd in \cite{BG1}, being first raised by Lubotzky, Philips and Sarnak \cite{Lub} in 1987 and is still open. In the setting of $SU_2$, our results can be viewed as addressing a quantitative version of a weak variant of this question.

\end{abstract}
\maketitle
\section{Introduction}
Let $G$ be a compact $n-$dimensional Lie group endowed with a left-invariant Riemannian metric $d$.  Thus \beqs \forall g, x, y \in G,\,d(x, y) = d(gx, gy).\eeqs 
We will denote by $C_G$ a constant depending on  $(G, d)$ that is greater than $1$. Suppose $g_1, \dots, g_k$ are chosen independently from the Haar measure on $G$. Let  $\A = \cup_{i \in [k]} \A_i$, where, $\A_i :=  \{g_i\} \cup \{g_i^{-1}\}$. 
Let the Heat kernel at $x$  corresponding to Brownian motion on $G$ with respect to the metric $d$ started at the origin $o \in G$ for time $t$ be $H_t(x)$.  Let $\mu_{\A}^\ell$ be the uniform measure over all words of length $\ell$ whose alphabets belong to $\A$. Our first result, Theorem~\ref{thm:main:1} relates to equidistribution and gives a lower bound on the probability that  $\|\mu_{\A}^\ell*H_t - \frac{\one}{\vol G}\|_{\LL^2(G)}$ is less than a specified quantity $2\eta$. Our second result, Theorem~\ref{thm:11} provides conditions under which the set of all elements of $G$ which can be expressed as words of length less or equal to $\ell$ with alphabets in $\A$, form a $2r-$net of $G$ with probability at least $1-\de$. For constant $\de$, both $k$ and $\ell$ can be chosen to be less than $C n \log (1/r)$, where $C$ is a universal constant.  Lastly, by analysing the situation when $G$ is an $n-$dimensional torus, we show in Section~\ref{sec:lower}, that the conditions stated in Theorem~\ref{thm:11} are nearly optimal.

The overall strategy employed here is similar to that of Landau and Russell in \cite{LR}, which reproves with better constants, the result of Alon and Roichman \cite{AR} that random Cayley graphs are expanders. However, in our context of compact Lie groups, the analysis requires additional ingredients, due to the fact that $\LL^2(G)$ is infinite dimensional. We deal with this difficulty by  restricting  our attention to certain finite dimensional subspaces of $\LL^2(G)$. These subspaces are defined as the spans of eigenfunctions of the Laplacian corresponding to eigenvalues that are less than certain finite values.

By a result of Dolgopyat \cite{Dolgo}, we know that if we pick two independent random elements from the Haar measure of a compact connected Lie group $G$, then the subgroup generated by these elements and their inverses is dense in $G$ almost surely. Unfortunately, the rate at which the random point set corresponding to words of a fixed length approaches $G$ in Hausdorff distance, as guaranteed by \cite{Dolgo} is far from the rate that one would obtain if the corresponding Markov operator had a spectral gap for its action on $\LL^2(G)$. For this reason, the results of this paper do not follow. For the case 
$G = SU_n$, Bourgain and Gamburd proved \cite{BG2} the existence of a spectral gap provided the entries of the generators are algebraic and the subgroup they generate is dense in $G$. There is a long line of work that this relates to, touching upon approximate subgroups and pseudorandomness, for which we direct the reader to the references in \cite{BG2}.
The question of a spectral gap when $G$ is $SU_2$ for random generators of the kind we consider was reiterated by Bourgain and Gamburd in \cite{BG1}, being first raised by Lubotzky, Philips and Sarnak \cite{Lub} in 1987 and is still open. In the setting of $SU_2$, our results can be viewed as addressing a quantitative version of a weak variant of this question.

\section{Analysis on a compact Lie group}



Suppose $F_1, F_2, \dots$ are eigenspaces of the Laplacian on $G$ corresponding to eigenvalues $0 = \la_0 < \la_1 <  \la_2 < \dots$. Let $f_i^1, \dots, f_i^j, \dots$ be an orthonormal basis for $F_i$, for each $i \in \N$.
$G$ acts on  functions in $\LL^2(G)$ via $T_g$, the translation operator,
$$T_g f (x) = f(g^{-1} x).$$
Thus each $F_i$ is a representation of $G$, though not necessarily an irreducible representation.

As stated in the introduction, let the Heat kernel at $x$  corresponding to Brownian motion on $G$ with respect to the metric $d$ started at the origin $o \in G$ for time $t$ be $H_t(x)$. When we wish to change the starting point for the diffusion, we will denote by $H(x, y, t)$ the probability density of Brownian motion started at $x$ at time zero ending at $y$ at time $t$.

The following is a theorem of Minakshisundaram and Pleijel \cite{Minakshisundaram1, Minakshisundaram}.
\begin{theorem}\lab{Minak}
For each $x \in G$  there is an asymptotic expansion
$$ H( x, x, t) \sim t^{-n/2}(a_0(x) + a_1(x)t + a_2(x)t^2 + \dots),$$ $t \ra 0$.
The $a_j$ are smooth functions on $G$.
\end{theorem}
Since $G$ is equipped with a left invariant metric, the $a_j(x)$ are constant functions.
We will use the following theorem of Grigoryan from \cite{Grigor}, where it appears as Theorem 1.1.

\begin{theorem}\lab{thm:grig}
Assume that for some points $x, y \in M$ and for all $t \in (0, T) $, 
$$p_t(x, x) \leq \frac{C_1}{\gamma_1(t)},$$ and $$p_t(y, y) \leq \frac{C_1}{\gamma_2(t)},$$ where $\gamma_1$ and $\gamma_2$ are increasing positive functions on $\mathbb{R}_+$ both satisfying
\beq \frac{\gamma_i(at)}{\gamma_i(t)} \leq A \frac{\gamma_i(as)}{\gamma_i(s)}\eeq for all $0 < t \leq s < T,$ for some constants $a, A > 1.$ Then for any $C > 4$ and all $t \in (0, T),$ 
\beq p_t(x, y) \leq \frac{C_2}{\sqrt{\gamma_1(\epsilon t)\gamma_2(\epsilon t)}}\exp\left(- \frac{d^2(x, y)}{Ct}\right)\eeq
for some $\epsilon = \epsilon(a, C) > 0.$

\end{theorem}

It follows from Theorem~\ref{Minak} that for some sufficiently small time $T>0$, we can choose $\gamma_1(t) = \gamma_2(t) = (\frac{1}2)t^{n/2}$ for $ t \in (0, T)$ in Theorem~\ref{thm:grig}.

This gives us the following corollary.

\begin{cor}\lab{thm:CLY}
For any constant $C > 4$, there exists $T>0$ and $C_1$ depending on $G$ and  $C$ so that for all $t \in (0, T)$
 $$H(x, y, t) \leq C_1t^{-n/2} \exp (-\frac{r^2}{Ct})$$
 where $n$ is the dimension of $G$ and $r$ is the distance between $x$ and $y$.
\end{cor}

\begin{lemma}\label{lem:1}Let $\eta > 0$.
We  take $ \eps  \sqrt{ 5\ln \frac{1}{\eta \eps^n}} = r$. 
If we choose $t = \eps^2$, then, 

 for all $y$ such that $$d(x, y) > r,$$ we have
$$H(x, y, t) < C_G{\eta}.$$
\end{lemma}
\begin{proof}
 In Corollary~\ref{thm:CLY},  we may set $C = 5$ and  $T = 1$ and ignore the dependence in $x$ since the metric is left invariant.  For all $ t \leq \eps^2$ and all $y$ such that $$d(x, y) >  r,$$
\beq H(x, y, t) & < &  C_1 \eps^{-n}  ( \exp(- \ln \frac{1}{\eta \eps^n})) \\
& < &                 C_1 {\eta}.\eeq

\end{proof}

By Weyl's law for the eigenvalues of the Laplacian on a Riemannian manifold as proven by Duistermaat and Guillimin \cite{DG}, we have the following.
\begin{theorem}\label{thm:weyl}
$$\lim_{\la \ra \infty} \frac{\la^{n/2}}{ \sum_{\la_i \leq \la}\dim F_i} = \frac{\vol(B_n)\vol(G)}{(2\pi)^n} =: C_2,$$ where $C_2$ is a constant depending only on volume and dimension $n$ of the Lie group.
\end{theorem}
This has the following corollary, which is improved upon by Theorem~\ref{thm:g} below.
\begin{corollary}\label{cor3}
$$\sup_{i \geq 1} \frac{\dim F_i}{ \la_i^{n/2}} = C_3,$$ where $C_3$ is a finite constant depending only on the Lie group and its metric.
\end{corollary}

The following theorem is due to Donnelly (Theorem 1.2, \cite{donnelly}). 

\begin{theorem}\lab{thm:g} 
Let $M$ be a compact $n-$dimensional Riemannian manifold and $\Delta$ its Laplacian acting on functions. Suppose that the injectivity radius of $\MM$ is bounded below by $c_4$ and that the absolute value of the sectional curvature is bounded above by $c_5$. If $\Delta \phi = - \la \phi$ and $\la \neq 0$, then $\|\phi\|_\infty \leq c_2 \la^{\frac{(n-1)}{4}}\|\phi\|_2.$ The constant $c_2$ depends only upon $c_4, c_5,$ and the dimension $n$ of $\MM$. Moreover the multiplicity $m_\la \leq c_3 \la^{\frac{(n-1)}{2}}$ where $c_3$ depends only on $c_2$ and an upper bound for the volume of $\MM$.
\end{theorem}
H\"{o}rmander \cite{Hormander} proved this result earlier without specifying which geometric parameters the constants  depended upon. 
Then, by the Fourier expansion of the heat kernel into eigenfunctions of the Laplacian,
$$H_t = \sum_{\la_i \geq 0}\sum_{j} a_{ij}  f_{ij} .$$
where $a_{ij} = e^{- \la_i t} f_{ij}(0) \leq  e^{- \la_i t} (c_2 \la_i^{\frac{n-1}{4}}),$ where the $f_{ij}$  for $ j \in [1, \dim F_i] \cap \N $, form an orthonormal basis of $F_i$.  
Let $$\tilde{H}_{t, M}(y) = \sum_{0 < \la_i \leq M}\sum_{j} a_{ij} f_{ij},$$ and
 $${H}_{t, M}(y) = \sum_{0 \leq \la_i \leq M}\sum_{j} a_{ij} f_{ij},$$
\begin{lemma}\lab{lem:6} For any $M > 0$, \beq \|\tilde{H}_{t, M}\|_{\LL^2} <  C_G t^{-n/4}\eeq
\end{lemma}
\begin{proof}
We note that \beq \|\tilde{H}_{t, M}\|_{\LL^2} \leq \|H_t\|_{\LL^2},\eeq because $\tilde{H}_{t, M}$ is the image of $H_t$ under a projection (with respect to $\LL^2$) onto a  subspace spanned by the eigenfunctions of the Laplacian corresponding to eigenvalues in the range $(0, M]$. Thus it suffices to bound $\|H_t\|_{\LL^2}$ from above in the appropriate manner. Choosing $\eta = 1 $ in Lemma~\ref{lem:1}, we see that if
we  take $ \eps  \sqrt{ 5 \ln (\eps^{-n})} = r$ and  $t = \eps^2$, then, 

 for all $y$ such that $$d(x, y) > r,$$ we have
$$H(x, y, t) < C_G.$$ Let $\mu_n$ denote the Lebesgue measure on $\R^n$ and $\mu$ the volume measure on $G$. We next need an upper bound on $\int_{B(o,r)} H_t(y)^2 \mu(dy).$ Note that when $\eps$ is sufficiently small, $B(o, r)$ is almost isometric via the exponential map to a Euclidean ball of radius $r$ in $\R^n$. 
Further, it is known that \beq \sqrt{\det g_{ij}(\exp_x(\a v))} = 1 - \frac{1}{6} Ric^g(v, v) \a^2 + o(\a^2), \eeq
where $Ric$ denotes the Ricci tensor, and $\exp_x$, the exponential map at $x$.

 Thus, 
 \beqs \int_{B(o,r)} H_t(y)^2 \mu(dy) & \leq &  C_G\left( \int_{\R^n} \eps^{-n}  ( \exp(-  \frac{|y|^2}{5t} ))\mu_n(dy)\right)\\
& \leq & C_G \left( \int_{\R} \eps^{-1}  ( \exp(-  \frac{|y|^2}{5t} ))\mu_1(dy)\right)^n\\
& \leq & C_G \eps^{-n}. \eeqs
Therefore \beq \|H_t\|_{\LL^2} \leq C_G \eps^{-n/2}. \eeq
\end{proof}



\begin{lemma}\lab{lem:7} For $M = 2^{\frac{2k_0}{n}}$ where $$k_0 \geq \max\left(\log_2 \frac{ 1 }{\eta},  C_G +  (1 + o(1))\frac{n}{2} \log_2 \frac{1}{t}\right),$$
 \beq\lab{eq:byW}   \|H_t - H_{t, M}\|_{\LL^2}  \leq {\eta}.\eeq \end{lemma}
\begin{proof}
It follows by the $\LL^2-$convergence of Fourier series that
\beq \|H_t - H_{t, M}\|_{\LL^2} \leq \sum_{\la_i \geq M} \dim(F_i)e^{- \la_i t} (c_2 \la_i^{\frac{n-1}{4}}) .\eeq

By Weyl's law (Theorem~\ref{thm:weyl}), 
$$\lim_{\la \ra \infty} \frac{\la^{n/2}}{ \sum_{\la < \la_i \leq 2^{\frac{2}{n}}\la}\dim F_i} = \frac{\vol(B_n)\vol(G)}{(2\pi)^n} =: C_2^{-1}.$$



Let, for $k \in \N$, \beq I_k = \left(2^{\frac{2k}{n}}, 2^{\frac{2k+2}{n}}\right]. \eeq Now, for $k_0 > C_G$, 
\beq  \sum_{\la_i > 2^{\frac{2k_0}{n}}} \dim(F_i)e^{- \la_i t} (c_2 \lambda_i^{\frac{n-1}{4}}) & \leq & 
\sum_{k \geq k_0} \left(\sum_{ \la_i \in I_k} \dim(F_i) \right) \sup_{\la_i \in I_k} \left(\frac{c_2 \la_i^{ \frac{n-1}{4}}}{e^{\la_i t}}\right)\\
& \leq & C_2 \sum_{k \geq k_0}  2^{k+1} \sup_{\la_i \in I_k} \left(\frac{c_2 \la_i^{ \frac{n-1}{4}}}{e^{\la_i t}}\right)
 \eeq 

We see that \beq \sup_{\la_i \in I_k} \left(\frac{\la_i^{\frac{n-1}{4}}}{e^{\la_i t}}\right) & < & \frac{2^{\frac{(k+1)}{2}}}{\exp(2^{\frac{2k}{n}}t)}\\
& < & \exp(\frac{(k+1)}{2} - 2^\frac{2k}{n} t).\eeq
When \beq\lab{eq:1.10}  k \geq \left(\frac{n}{2}\right) \log_2 \frac{6k}{t}, \eeq  assuming $k > 5$, we have \beq \frac{k/t}{n/2t} \geq \log_2 \frac{\frac{5}{2}(k+1)}{t}, \eeq and then, we see that 
\beq \exp(\frac{(k+1)}{2} - 2^\frac{2k}{n} t) < 2^{-2( k+1)}. \eeq
In order to enforce (\ref{eq:1.10}), it suffices to have \beq \frac{k}{\log_2 \frac{6k}{t}} & \geq &  \frac{n}{2},\eeq which is implied by \beq \frac{6k}{\log_2 \frac{6k}{t}}\log_2\left( \frac{6k}{t\log_2 \frac{6k}{t}} \right) & \geq &  {3n}\log_2\left(\frac{3n}{t}\right).\eeq This is equivalent to 
\beq k\left(1 - \frac{\log_2 {\log_2 \frac{6k}{t}}}{{\log_2 \frac{6k}{t}}} \right) & \geq &  \frac{n}{2}\log_2\left(\frac{3n}{t}\right),\eeq
which is in turn implied by 
\beq k & \geq &  \frac{n}{2} \left(\log_2 \frac{3n}{t}\right)\left( 1 - \frac{\log_2 \log_2 \frac{3n}{t}}{\log_2 \frac{3n}{t}}\right)^{-1}\\ &  =  & (1 + o(1))\frac{n}{2} \log_2 \frac{3n}{t}. \eeq

Therefore, for any $$k_0 > C_G +  (1 + o(1))\frac{n}{2} \log_2 \frac{3n}{t},$$
\beq \sum_{\la_i > 2^{\frac{2k_0}{n}}} \dim(F_i)e^{- \la_i t} (c_2 \lambda_i^{\frac{n-1}{4}})  < \frac{2^{(-k_0-1)}}{1 - (1/2)} < 2^{(-k_0)}. \eeq

\end{proof}
 It follows from (\ref{eq:byW}) that for any $\eta$, by choosing $$k_0 = \max\left(\log_2 \frac{1}{\eta}, C_G +  (1 + o(1)){n} \log_2 \frac{1}{\eps}\right),$$ and \beq\lab{eq:M} M \geq  2^{2k_0/n}\eeq 
we have that
\beq \label{eq:one}  \|\tilde{H}_{t, M}  - {H}_t\|_{\LL^2} < \eta.\eeq

\section{Equidistribution and an upper bound on the Hausdorff distance.}
Let $A(V)$ denote the collection of self adjoint operators on the finite dimensional Hilbert space $V$. For $B \in A(V)$, we let $\|B\|$ denote the operator norm of $B$, equal to the largest absolute value attained by an eigenvalue of $A$. The cone of {\it non-negative definite} operators
$$\Lambda(V) = \{B\in A(V)|\forall v, \langle Av, v\rangle \geq 0\}$$ turns $A(V)$ into a poset by the relation $A \geq B$ if $A - B \in \Lambda(V)$.

We next state a matrix Chernoff bound due to Ahlswede and Winter from \cite{AW}.
 \begin{theorem}\lab{thm:AW}
Let $V$ be a Hilbert Space of dimension $D$ and let $A_1, \dots, A_k$ be independent identically distributed random variables taking values in $\Lambda(V)$ with expected value $\E[A_i] = A\geq \mu I$ and $A_i \leq I$. Then for all $\eps \in [0, 1/2]$,
$$ \p\left[\frac{1}{k} \sum_{i=1}^k A_i \not\in [(1-\epsilon)A, (1 + \epsilon)A]\right] \leq 2D \exp\left(\frac{-\eps^2\mu k}{2\ln 2}\right).$$
\end{theorem}

For any $g \in G$
\beq (Id - T_g) \tilde{H}_{t, M} \eeq lies in \beq \tilde{F}_M := \bigoplus_{ 0 < \la_i \leq M} F_i. \eeq
$\tilde{F}_M$ has, by Weyl's law, a  dimension that is bounded above by $O(M^{n/2})$.
We will study the Markov operator $P: \tilde{F}_M \To \tilde{F}_M$ given by
\beq P(f)(x) := \frac{\sum\limits_{g \in \A} (f(x) + f(gx))}{2|\A|}. \eeq

We know that  $\A = \cup_i \A_i$, where, $\A_i =  \{g_i\} \cup \{g_i^{-1}\}$.
Note that $P$ is the sum of $k$ i.i.d operators \beq P_i :=  \frac{\sum\limits_{g \in \A_i} (f(x) + f(gx))}{4}. \eeq
We see that $\forall f \in \tilde{F}_M$, and $1 \leq i \leq k$,  \beq \E P_i(f) = (1/2) f, \eeq which is equivalent to
$$\E P_i = (1/2) I. $$

By Theorem~\ref{thm:AW}, 
for all $\eps \in [0, 1/2]$,
\beq \p\left[\frac{1}{k} \sum_{i=1}^k P_i \not\in [((1-\epsilon)/2)I, ((1 + \epsilon)/2) I]\right] \leq C_GM^{n/2} \exp\left(\frac{-\eps^2 k}{4\ln 2}\right).\eeq Setting $\eps = 1/2$ and substituting for $M$, we see that 

\beq \p\left[\frac{1}{k} \sum_{i=1}^k P_i \not\in [(1/4)I, (3/4) I]\right] \leq \left(C_G M^{n/2}\right) \exp\left(\frac{- k}{16\ln 2}\right).\eeq
Let the map $x \mapsto gx$ be denoted by $T_g$. 
It follows that  
\beqs \p\left[\forall f \in \tilde{F}_M, \left\|\frac{1}{2k} {\sum\limits_{g \in \A}  f \circ T_g} \right\|_{\LL^2} \leq (1/2)\|f\|_{\LL^2} \right] \geq 1 - \left(C_G M^{n/2}\right) \exp\left(\frac{- k}{16\ln 2}\right).\eeqs
Iterating the above inequality $\ell$ times, we observe that
\beqs \p\left[\forall f \in \tilde{F}_M, \left\|\frac{1}{(2k)^\ell} {\sum\limits_{g \in \A^\ell}  f \circ T_g} \right\|_{\LL^2} \leq (1/2)^\ell \|f\|_{\LL^2} \right] \geq 1 - \de,\eeqs
where \beq\lab{eq:de} \de := \left(C_G M^{n/2}\right) \exp\left(\frac{- k}{16\ln 2}\right).\eeq

Choosing $f = \tilde{H}_{t, M}$, we see that
\beqs \p\left[ \left\|\frac{1}{(2k)^\ell} {\sum\limits_{g \in \A^\ell}   \tilde{H}_{t, M} \circ T_g} \right\|_{\LL^2} \leq (1/2)^\ell \| \tilde{H}_{t, M}\|_{\LL^2} \right] \geq 1 - \de,\eeqs
By the above, and Lemmas~\ref{lem:6} and \ref{lem:7}, we see that
\beqs \p\left[ \left\|\frac{1}{(2k)^\ell} {\sum\limits_{g \in \A^\ell}   \tilde{H}_{t} \circ T_g} \right\|_{\LL^2} \leq \eta + 2^{-\ell} t^{-n/4} \right] \geq 1 - \de.\eeqs
Thus, we see that

\beq \lab{eq:eta} \p\left[ \left\|\frac{1_G}{\vol G} - \frac{1}{(2k)^\ell} {\sum\limits_{g \in \A^\ell}   {H}_{t} \circ T_g} \right\|_{\LL^2} \leq \eta + 2^{-\ell} t^{-n/4} \right] \geq 1 - \de.\eeq

We derive from this, the following theorem on the equidistribution of $\A^\ell$. 
\begin{theorem}\lab{thm:main:1}
Let $2^{- \ell} t^{- \frac{n}{4}} \leq \eta \leq 2^{-C_G}t^{\frac{(1 + o(1))n}{2}}.$ Let $\de =  (C_G/\eta)\exp\left(-\frac{k}{16\ln 2}\right).$ Then, 

\beq \lab{eq:eta1} \p\left[ \left\|\frac{1_G}{\vol G} - \frac{1}{(2k)^\ell} {\sum\limits_{g \in \A^\ell}   {H}_{t} \circ T_g} \right\|_{\LL^2} \leq 2\eta  \right] \geq 1 - \de.\eeq

\end{theorem}

\begin{proof}
This follows from (\ref{eq:eta}) on setting $M = \eta^{-\frac{2}{n}}$ and substituting in (\ref{eq:de}). 
\end{proof}

\begin{lemma}\lab{lem:9}
Suppose $ \eps  \sqrt{ 5\ln \frac{C_G}{ \eps^n}} = r$, and $t = \eps^2$ are sufficiently small.
If $$\left\|\frac{1_G}{\vol G} - \frac{1}{(2k)^\ell} {\sum\limits_{g \in \A^\ell}   {H}_{t} \circ T_g} \right\|_{\LL^2} \leq \sqrt{\vol(B_n)r^{n}}\left(\frac{1}{2\vol(G)}\right),$$ then,  $\A^\ell$ is a $2r$-net of $G$.
\end{lemma}
\begin{proof}
Suppose $\A^\ell$ is not a $2r$-net of $G$. Then, there exists an element $\tilde{g}$ such that $d(\tilde{g}, \A^\ell) > 2r$.
Let $B(r, \tilde{g})$ be the metric ball of radius $r$ centered at $\tilde{g}$. Then, for any $g \in \A^\ell$, $B(r, g) \cap B(r, \tilde{g}) = \emptyset.$ Applying  Lemma~\ref{lem:1}  we see that 
${H}_{t}(g^{-1} y) < \frac{1}{3\vol G}$ for all $g \in \A^\ell$ and all $y \in B(r, \tilde{g})$. Therefore, 
$$       \frac{1}{(2k)^\ell} {\sum\limits_{g \in \A^\ell}   {H}_{t} \circ T_g(y)}             < \frac{1}{3\vol G}$$ for all all $y \in B(r, \tilde{g})$. This implies that 
\beq \left\|\frac{1_G}{\vol G} - \frac{1}{(2k)^\ell} {\sum\limits_{g \in \A^\ell}   {H}_{t} \circ T_g} \right\|_{\LL^2} & > &
\sqrt{\vol(B(0, r))}\left(\frac{2}{3\vol(G)}\right)\\ & > & \sqrt{\vol(B_n)r^{n}}\left(\frac{1}{2\vol(G)}\right),\eeq
which is a contradiction.
\end{proof}

\begin{theorem} \lab{thm:11} Suppose $ \eps  \sqrt{ 5\ln \frac{C_G}{ \eps^n}} = r$.
Choose $$k \geq C_G + (16 \ln 2)((1 + o(1))n\ln \frac{1}{\eps} +  \ln \frac{1}{\de})$$ i.i.d random points $\{g_1, \dots, g_k\}$ from the Haar measure on $G$ and let $$\A = \{g_1, g_1^{-1}, \dots, g_k, g_k^{-1}\}.$$ Let $X$ be the set of all elements of $G$ which can be expressed as words of length less or equal to $\ell$ with alphabets in $\A$, where $\ell \geq C_G+ \frac{n}{2}\log_2(\frac{1}{\eps r})$ . Then, with probability at least $1 - \de$, for every element $g \in G$ there is $x \in X$ such that $d(g, x) < 2r$.
\end{theorem}
\begin{proof}
Let $\eta = 2^{- C_G} \eps^{(1 + o(1))n}$ in Lemma~\ref{lem:7}. We set $\log_2 M = C_G + \log_2 \frac{1}{t^{1 + o(1)}},$ by enforcing an equality in (\ref{eq:M}). Taking logarithms on both sides of (\ref{eq:de}), we see that 
$$- \ln \frac{1}{\de} = C_G + \frac{n}{2} \ln t^{-(1 + o(1))} - \frac{k}{16\ln 2}.$$ This fixes the lower bound for $k$ in the statement of the corollary. 
In order to use (\ref{eq:eta}) in conjunction with Lemma~\ref{lem:9}, we see that it suffices to set $2^{-\ell}t^{- \frac{n}{4}}$ to a value less than $r^{n/2}$, because for small $\eps$, the value of $\eta$ that we have chosen is significantly smaller than $r^{n/2}$. This shows that the theorem holds for any $\ell$ greater or equal to  $ \frac{n}{2} \log_2 \frac{1}{\eps r} + C_G.$
\end{proof}

\section{Lower bounds on the number of generators and  word length required for a given Hausdorff distance}\lab{sec:lower}

Let us take $\eps > 0$, $\de > 0$ and denote  $2k + \ell =: m$. We are now interested in lower bounds on the values of $k$ and $\ell$ as a function on $\eps, \de$ and $m$. 

In this section, we consider only the case when $G = \R^n/\Z^n,$  the unit $n-$dimensional torus.

The number of distinct elements that the set of all words $\A^\ell$ of length $\ell$ in $k$ alphabets (and their inverses) correspond to in an abelian group is less or equal to 
${2k + \ell - 1\choose \ell}$. Therefore, the measure covered by a $2r-$neighborhood (in an $\ell_2^n$ sense) of ${2k + \ell - 1\choose \ell}$ elements is less or equal to ${m - 1\choose \ell}(2r)^n \vol(B_n),$ for $r < 1/2$. Since this has to be at least $1$ for $\A^\ell$ to be a $2r-$net of $G$, we obtain 
\beq {m - 1 \choose \ell}(2r)^n \vol(B_n) \geq 1. \eeq


We will obtain a lower bound for $k$ using the above inequality. It is clear that a lower bound that is twice as big must hold for $\ell+1$ by the identity 

$${2k + \ell - 1\choose \ell} = {2k + \ell - 1\choose 2k-1}.$$

\beq {m - 1 \choose \ell}2^n \vol(B_n) \geq \frac{1}{r^n}, \eeq

therefore, \beq \ln {m^{2k}} - \ln {(2k-1)!} + \ln\left(2^n \vol(B_n)\right) \geq n\ln \frac{1}{r}. \eeq
Therefore \beq k \geq \frac{ n \ln \frac{1}{r}}{2 \ln m}.\eeq
Let us place the further constraint that $\ell \leq k \leq 2\ell$ as was the case in our upper bound for reasonably large values of $\de$.
We then using Stirling's approximation see that $$2 k\ln 3k - (2k -2) \ln (2k) + 2k  \geq { n \ln \frac{1}{r}} -  \ln\left(2^n \vol(B_n)\right).$$
This leads to 
$$2 \ln 2 k + 2k (1 + \ln (3/2))   \geq {  n \ln \frac{1}{r}}  - \ln\left(2^n \vol(B_n)\right).$$

Therefore, for sufficiently large $n$, $$k  \geq  (2\ln(1 + \ln (3/2)))^{-1}\left(n \ln \frac{1}{r} -  \ln\left( n \ln \frac{1}{r}\right)  - \ln\left(2^n \vol(B_n)\right) -  \ln\left(-\ln\left(2^n \vol(B_n)\right)\right)\right).$$
This can be simplified  to the following weaker inequality: 
\beq k  \geq  (2\ln(1 + \ln (3/2)))^{-1}\left(n \ln \frac{1}{r} -  \ln\left(  \ln \frac{1}{r}\right)  \right).\eeq Therefore, we also have
\beq \ell + 1 \geq  (\ln(1 + \ln (3/2)))^{-1}\left(n \ln \frac{1}{r} -  \ln\left(  \ln \frac{1}{r}\right)  \right).\eeq

\section{Acknowledgements}
We are grateful to Charles Fefferman, Anish Ghosh, Sergei Ivanov and Matti Lassas for helpful discussions. We thank
Emmanuel Breuillard for a useful correspondence and directing us to \cite{Dolgo}. We are grateful to Somnath Chakraborty for a careful reading and numerous corrections. This work was supported by NSF grant  \#1620102  and a Ramanujan fellowship.





\bibliographystyle{amsplain}

\end{document}